\documentclass[11pt]{amsart}
\usepackage{amssymb}
\usepackage[mathscr]{euscript}

\newcommand{\mbf}{\mathbf}                    
\newcommand{\msc}{\mathscr}

\usepackage{graphicx}

\newtheorem{thm}{Theorem} 

\newtheorem{lm}[thm]{Lemma}
\newcommand{\la}{\langle}
\newcommand{\ra}{\rangle}
\newcommand{\op}{\operatorname}
\newcommand{\join}{\vee}
\newcommand{\meet}{\wedge}
\renewcommand{\Join}{\bigvee}
\newcommand{\Meet}{\bigwedge}

\newcommand{\id}{\downarrow\!}
\newcommand{\fil}{\uparrow\!}
\newcommand{\vare}{\varepsilon}
\newcommand{\ol}{\mathbf}
\newcommand{\ifff}{\Longleftrightarrow}
\newcommand{\ssm}{\smallsmile}

\newcommand{\A}{\mathbf A}
\newcommand{\B}{\mathbf B}
\newcommand{\C}{\mathbf C}
\newcommand{\E}{\mathscr E}
\newcommand{\F}{\mathbf F}
\renewcommand{\H}{\mathscr H}

\renewcommand{\L}{\mathbf L}
\newcommand{\M}{\mathbf M}
\renewcommand{\S}{\mathbf S}
\newcommand{\T}{\mathbf T}
\newcommand{\V}{\mathscr V}

\date{\today}
\keywords{quasivariety, quasi-equational theory, congruence lattice, 
semilattice}
\subjclass[2010]{08C15, 08A30, 06A12}

\begin{document}

\title[Q-lattices as congruence lattices]%
{Lattices of Quasi-equational Theories as Congruence Lattices of 
Semilattices with Operators, Part I}

\author {Kira Adaricheva}
\address{Department of Mathematical Sciences, Yeshiva University,
New York, NY 10016, USA}
\email{adariche@yu.edu}

\author {J. B. Nation}
\address{Department of Mathematics, University of Hawaii, Honolulu, HI
96822, USA}
\email{jb@math.hawaii.edu}

\thanks{The authors were supported in part by a grant from the U.S.~
Civilian Research \& Development Foundation.  The first author was also
supported in part by INTAS Grant N03-51-4110.}

\begin{abstract}
We show that for every quasivariety $\msc K$ of structures 
(where both functions and relations are allowed)
there is a semilattice $\S$ with operators such that the lattice of
quasi-equational theories of $\msc K$ (the dual of the lattice of
sub-quasivarieties of $\msc K$) is isomorphic to $\op{Con}(\S,+,0,\msc F)$.
As a consequence, new restrictions on the natural quasi-interior operator
on lattices of quasi-equational theories are found.
\end{abstract}

\maketitle

\section{Motivation and terminology}

Our objective is to provide, for the lattice of quasivarieties contained
in a given quasivariety (\emph{$Q$-lattices} in short), a description similar 
to the one that characterizes the lattice of subvarieties of a given variety 
as the dual of the lattice of fully invariant congruences on a countably 
generated free algebra.  Just as the result for varieties is more naturally 
expressed in terms of the lattice of equational theories, rather than the 
dual lattice of varieties, so it will be more natural to consider lattices 
of quasi-equational theories rather than lattices of quasivarieties.

The basic result is that the lattice of quasi-equational theories extending
a given quasi-equational theory is isomorphic to the congruence lattice
of a semilattice with operators preserving join and $0$.  These lattices
support a natural quasi-interior operator, the properties of which lead 
to new restrictions on lattices of quasi-equational theories.

This is the first paper in a series of four.  Part~II shows that if $\S$
is a semilattice with both $0$ and $1$, and $\msc G$ is a \emph{group} of
operators on $\S$ such that each operator in $\msc G$ fixes both $0$ and $1$,
then there is a quasi-equational theory $\msc T$ such
that $\op{Con}(\S,+,0,\msc G)$ is isomorphic to the lattice of quasi-equational
theories extending $\msc T$.  The third part~\cite{JBN08} shows that if
$\S$ is any semilattice with operators, then $\op{Con} \S$ is isomorphic to
the lattice of implicational theories extending some given implicational
theory, but in a language that may not include equality.
The fourth paper~\cite{HKNT}, with T.~Holmes, D.~Kitsuwa and S.~Tamagawa, 
concerns the structure of lattices of atomic theories in a language without 
equality.

The setting for varieties is traditionally \emph{algebras}, i.e., sets with
operations, whereas work on quasivarieties normally allows 
\emph{structures}, i.e., sets with operations and relations.  
Some adjustments are required for the more general setting.
Let us review the universal algebra of structures, following
Section 1.4 of Gorbunov \cite{VG2}; 
see also Gorbunov and Tumanov \cite{GT1, GT2} and Gorbunov \cite{VG3}.

The \emph{type} of a structure is determined by its \emph{signature}
$\sigma = \la \msc F, \msc R, \rho \ra$ where
$\msc F$ is a set of function symbols,
$\msc R$ is a set of relation symbols,
and $\rho:\msc F \cup \msc R \to \omega$ assigns arity.
A \emph{structure} is then $\A = \la A,\msc F^{\A},\msc R^{\A} \ra$
where $A$ is the carrier set, 
$\msc F^{\A}$ is the set of operations on $A$,
and $\msc R^{\A}$ is the set of relations on $A$.

For structures $\A$ and $\B$ of the same type,
a map $h: \A \to \B$ is a \emph{homomorphism} if it preserves
operations and $h(R^{\A}) \subseteq R^{\B}$ for each relation
symbol $R$.  An \emph{endomorphism} of $\A$ is a homomorphism
$\vare: \A \to \A$.   

The \emph{kernel} $\ker h$ of a homomorphism $h$
is a pair $\kappa = \la \kappa_0, \kappa_1 \ra$ where
\begin{itemize}
\item $\kappa_0$ is the equivalence relation on $A$ induced by $h$,
i.e., $(x,y) \in \kappa_0$ iff $h(x)=h(y)$, 
\item $\kappa_1 = \bigcup_{R \in \msc R} \kappa_1^R$ where
$\kappa_1^R = h^{-1}(R^{\B}) = 
\{ \ol s \in A^{\rho(R)}: h(\ol s) \in R^{\B} \}$.
\end{itemize}
Equality is treated differently because, in standard logic,
equality is assumed to be a congruence relation.  
Indeed, the statements that $\approx$ is reflexive, symmetric, transitive,
and compatible with the functions of $\msc F$ and the relations of $\msc R$,
are universal Horn sentences.  Thus in normal quasi-equational logic we are
working in the quasivariety given by these laws.
This is not necessary:  see Parts III and IV~\cite{JBN08, HKNT}.

A \emph{congruence} on a structure $\mbf A = \la A, \msc F^{\mbf A},
\msc R^{\mbf A} \ra$ is a pair $\theta = \la \theta_0, \theta_1 \ra$ where
\begin{itemize}
\item $\theta_0$ is an equivalence relation on $A$ that is compatible with
the operations of $\msc F^{\mbf A}$, and
\item $\theta_1 = \bigcup_{R \in \msc R} \theta_1^R$ where each
$\theta_1^R \subseteq  A^{\rho(R)}$ and
$R^{\mbf A} \subseteq \theta_1^R$, i.e., the original relations
of $\mbf A$ are contained in those of $\theta_1$, and for each
$R \in \msc R$, if $\ol a \in \theta_1^R$ and $\ol b \in A^{\rho(R)}$ and
$\ol a \,\theta_0\, \ol b$ componentwise, then $\ol b \in \theta_1^R$.
\end{itemize}
Note that if $h:\A \to \B$ is a homomorphism, then $\ker h$ is a
congruence on $\A$.
The collection of all congruences on $\A$ forms an algebraic lattice 
$\op{Con}\,\A$ under set containment.

A subset $S \subseteq A$ is a \emph{subuniverse} if it is closed
under the operations of $\A$.  A \emph{substructure} of $\A$ is
$\S = \la S, \msc F^{\S}, \msc R^{\S} \ra$ where $S$ is a subuniverse
of $A$, 
for each operation symbol $f \in \msc F$ the operation $f^{\S}$ is
the restriction of $f^{\A}$ to $S^{\rho(f)}$, and
for each relation symbol $R \in \msc F$ the relation $R^{\S}$ is
$R^{\A} \cap S^{\rho(R)}$. 

Given a congruence $\theta$ on a structure $\A$, we can form a
\emph{quotient structure} $\A/\theta$ by defining operations
and relations on the $\theta_0$-classes of $A$ in the natural way.
The isomorphism theorems carry over to this more general setting.
In particular, if $h:\A \to \B$ is a homomorphism, then $h(\A)$ is
a substructure of $\B$, and $h(\B)$ is isomorphic to $\A/\ker h$.

A congruence is \emph{fully invariant} if, for every endomorphism 
$\vare$ of $\A$,  
\begin{itemize}
\item $a \,\theta_0\, b$ implies $\vare(a) \,\theta_0\, \vare(b)$, and
\item for each $R \in \msc R$, 
$\ol a \in \theta_1^R$ implies $\vare (\ol a) \in \theta_1^R$.
\end{itemize}
The lattice of fully invariant congruences is denoted $\op{Ficon}\,\A$.

The congruence generation theorems are straightforward to generalize.
Let $C \subseteq A^2$ and let $D$ be a set of formulae of the form
$R(\mbf a)$ with $R \in \msc R$ and $\mbf a \in A^{\rho(R)}$.
The congruence generated by $C \cup D$, denoted $\op{con}(C \cup D)$, 
is the least congruence $\theta = \la \theta_0,\theta_1 \ra$ such that
$C \subseteq \theta_0$  and $\mbf a \in \theta_1^R$ for all 
$R(\mbf a) \in D$.  
The equivalence relation $\theta_0$ is given by the usual Mal'cev 
construction 
applied to $C$, and $\theta_1$ is the closure of $D \cup \msc R^{\A}$ 
with respect to $\theta_0$, i.e., if $R(\mbf a) \in D$ and 
$\mbf a \,\theta_0\, \mbf b$ componentwise, then $\mbf b \in \theta_1^R$.

A \emph{variety} is a class closed under homomorphic images, substructures
and direct products.  Varieties are determined by laws of the form 
$s \approx t$ and $R(\ol s)$ where $s$, $t$ and the components of 
$\ol s$ are terms.  That is, a variety is the class of all similar 
structures satisfying a collection of atomic formulae.  If $\msc V$ is a
variety of structures and $\mbf F$ is the countably generated
free structure for $\msc V$, then the lattice $\op{L}_v(\msc V)$ of
subvarieties of $\msc V$ is dually isomorphic to the lattice of fully
invariant congruences of $\mbf F$, i.e., $\op{L}_v(\msc V) \cong^d
\op{Ficon}\,\mbf F$.  In the case of varieties of algebras (with no
relational symbols in the language), this is equivalent to adding the
endomorphisms of $\mbf F$ to its operations and taking the usual congruence
lattice, so that
$\op{L}_v(\msc V) \cong^d \op{Con}\, (F,\msc F \cup \op{End}\,\mbf F)$.
For structures in general, this simplification does not work.
(These standard results are based on Birkhoff \cite{GB}.)

A \emph{quasivariety} is a class of structures closed under
substructures, direct products and ultraproducts (equivalently, substructures
and reduced products).  Quasivarieties are
determined by laws that are \emph{quasi-identities}, i.e., Horn sentences
\[ \&_{1 \leq i \leq n} \alpha_i \implies \beta \]
where the $\alpha_i$ and $\beta$ are atomic formulae of the form
$s \approx t$ and/or $R(\ol s)$.

If $\msc K$ is a quasivariety and $\A$ a structure, then a congruence
$\theta$ on $\A$ is said to be a {$\msc K$-congruence} if 
$\A/\theta \in \msc K$.  Since the largest congruence is a $\msc K$-congruence,
and $\msc K$-congruences are closed under intersection, the set of
$\msc K$-congruences on $\A$ forms a complete meet subsemilattice of
$\op{Con}\,\A$, denoted $\op{Con}_{\msc K}\,\A$.  
Moreover, $\op{Con}_{\msc K}\,\A$ is itself an algebraic lattice. 

Let us adopt some notation to reflect the standard duality between
theories and models.  For a variety $\msc V$, let $\op{ATh}(\msc V)$
denote the lattice of ``equational" (really, atomic) theories extending
the theory of $\msc V$, so that $\op{ATh}(\msc V) \cong^d \op{L}_v(\msc V)$.
Likewise, for a quasivariety $\msc K$, let $\op{QTh}(\msc K)$
denote the lattice of quasi-equational theories containing
the theory of $\msc K$, so that $\op{QTh}(\msc K) \cong^d \op{L}_q(\msc K)$.

Gorbunov and Tumanov described the lattice $\L_q(\msc K)$ of quasivarieties
contained in a given quasivariety $\msc K$ in terms of algebraic subsets.
This description requires some definitions.
\begin{itemize}
\item Given $\msc K$, let $\mbf F = \F_{\msc K}(\omega)$ be the countably
generated $\msc K$-free structure.
Then $\op{Con}_{\msc K}\,\F$ denotes the lattice of all
$\msc K$-congruences of $\mbf F$.
\item Define the isomorphism relation $I$ and embedding relation $E$ on
$\op{Con}_{\msc K}\,\F$ by
\begin{align*}
\varphi \,I\, \psi \quad&\text{if}\quad \F/\psi \cong \F/\varphi \\
\varphi \,E\, \psi \quad&\text{if}\quad \F/\psi \leq \F/\varphi .
\end{align*}
\item For a binary relation $R$ on a complete lattice $\mbf L$, let
$\op{Sp}(\L,R)$ denote the lattice of all $R$-closed algebraic subsets
of $\L$.  (Recall that $S \subseteq L$ is \emph{algebraic} if it is closed
under arbitrary meets and nonempty directed joins.  The set $S$ is 
\emph{$R$-closed} if $s \in S$ and $s \,R\, t$ implies $t \in S$.)
\end{itemize}
The characterization theorem of Gorbunov and Tumanov \cite{GT2} then says that
\[ \op{L}_q(\msc K) \cong \op{Sp}(\op{Con}_{\msc K} \, \F,I)
 \cong \op{Sp}(\op{Con}_{\msc K} \, \F,E) .\]
See Section~5.2 of Gorbunov~\cite{VG2}; also cf.~Hoehnke \cite{HJH}. 

By way of comparison, we might say that the description of the lattice
of subvarieties by $\op{L}_v(\msc V) \cong^d \op{Ficon}\,\F$ reflects
equational logic, whereas the representation 
$\op{L}_q(\msc K) \cong \op{Sp}(\op{Con}_{\msc K} \, \mbf F,E)$ say reflects
structural properties (closure under $\op{S}$, $\op{P}$ and direct limits).
We would like to find an analogue of the former for quasivarieties,
ideally something of the form $\op{L}_q(\msc K) \cong^d \op{Con}\,\S$
for some semilattice $\S$ with operators, reflecting quasi-equational logic.
This is done below.
Indeed, while our emphasis is on the structure of $Q$-lattices,
Bob Quackenbush has used the same general ideas to provide a nice algebraic
proof of the completeness theorem for quasi-equational logic~\cite{RWQ}.

The lattice $\op{QTh}(\msc K)$ of theories of a quasivariety is 
algebraic and (completely) meet semidistributive.  
Most of the other known properties of these lattices can be described 
in terms of the natural equa-interior operator, which is the dual of 
an equational closure operator on $\op{QTh}(\msc K)$.  See Appendix~II
or Section~5.3 of Gorbunov~\cite{VG2}.

A.M. Nurakunov~\cite{AMN}, building on earlier work of R.~McKenzie~\cite{RMcK}
and R.~Newrly ~\cite{NN},
has recently provided a nice algebraic description of the lattices
$\op{ATh}(\msc V)$, where $\msc V$ is a variety of algebras,
as congruence lattices of monoids with two additional unary operations
satisfying certain properties.  See Appendix~III.

Finally, let us note two (related) major differences between quasivarieties
of structures \emph{versus} algebras.  Firstly, the greatest quasi-equational 
theory in $\op{QTh}(\msc K)$ need not be
compact if the language of $\msc K$ has infinitely many relations.
Secondly, many nice representation theorems for quasivarieties use
one-element structures, whereas one-element algebras are trivial.
Indeed, in light of Theorem~\ref{consl} below, Theorem~5.2.8 of 
Gorbunov~\cite{VG2} (from Gorbunov and Tumanov~\cite{GT0}) can be
stated as follows.

\begin{thm} \label{one-elt}
The following are equivalent for an algebraic lattice $\L$.
\begin{enumerate}
\item $\L \cong \op{Con}(\S,+,0)$ for some semilattice $\S$.
\item $\L \cong \op{QTh}(\msc K)$ for some quasivariety $\msc K$ of 
one-element structures.
\end{enumerate}
\end{thm}

Congruence lattices of semilattices are coatomistic, i.e., every element 
is a meet of coatoms.  Thus the $Q$-lattices for the special quasivarieties
in the preceding theorem are correspondingly atomistic.

\section{Congruence lattices of semilattices}

Let $\op{Sp}(\L)$ denote the lattice of algebraic subsets of a complete
lattice $\L$.   If $\L$ is an algebraic lattice, let $\L_c$ denote its
semilattice of compact elements.  This is a join semilattice with zero.
The following result of Fajtlowicz and Schmidt \cite{FS} directly generalizes 
the Freese-Nation theorem \cite{FN}.
See also \cite{FKN}, \cite{HMS}, \cite{ETS}.

\begin{thm} \label{consl}
If\/ $\L$ is an algebraic lattice, then $\op{Sp}(\L) \cong^d \op{Con}\,\L_c$.
\end{thm}

\begin{proof}
For an arbitrary join 0-semilattice $\S = \la S,+,0 \ra$ we set up a Galois
correspondence between congruences of $\S$ and algebraic subsets of
the ideal lattice $\msc I(\S)$ as follows.

For $\theta \in \op{Con}\,\S$, let $h(\theta)$ be the set of all
$\theta$-closed ideals of $\S$.

For $\H \in \op{Sp}(\msc I (\S))$, let $x \  \rho(\H) \ y$ if
$\{ I \in \H: x \in I \} = \{ J \in \H: y \in J \}$. 

It is straightforward to check that $h$ and $\rho$ are order-reversing,
that $h(\theta) \in \op{Sp}(\msc I(\S))$ and $\rho(\H) \in \op{Con}\,\S$.

To show that $\theta = \rho h(\theta)$, we note that if $x < y$ (w.l.o.g.)
and $(x,y) \notin \theta$, then $\{ z \in S : x+z \, \theta \, x \}$ is
a $\theta$-closed ideal containing $x$ and not $y$.  Hence $(x,y) \notin
\rho h(\theta)$.

To show that $\H = h \rho(\H)$, consider an ideal $J \notin \H$.
For any $x \in S$, let $\hat{x} = \bigcap \{ I \in \H : x \in I \}$,
noting that $\hat{x} \in \H$.  Then $\{ \hat{x} : x \in J \}$ is
up-directed, whence $\bigcup \{ \hat{x} : x \in J \} \in \H$.
Therefore the union properly contains $J$, so that there exist $x<y$
with $x \in J$ and $y \in \hat{x}-J$, and $J$ is not $\rho(\H)$-closed.  
Thus $J \notin \H$ implies $J \notin h \rho(\H)$, as desired.
\end{proof}

Compare this with the following result of Adaricheva, Gorbunov and Tumanov
(\cite{AGT} Theorem~2.4, also \cite{VG2} Theorem~4.4.12).

\begin{thm}
Let\/ $\L$ be a join semidistributive lattice that is finitely presented
 within the class $\mbf{SD}_{\join}$.  Then $\L \leq \op{Sp}(\A)$
for some algebraic and dually algebraic lattice $\A$.
\end{thm}

On the other hand, Example 4.4.15 of Gorbunov \cite{VG2} gives a 4-generated
join semidistributive lattice that is not embeddable into any lower
continuous lattice satisfying $\op{SD}_{\join}$.

Keith Kearnes points out that the class $\msc{ES}$ of lattices that are
embeddable into congruence lattices of semilattices is not first order.
Indeed, every finite meet semidistributive lattice is in $\msc {ES}$, 
and $\msc {ES}$ is closed under $\op{S}$ and $\op{P}$.
Now the quasivariety $\op{SD}_{\meet}$ is generated by its finite members
(Tumanov \cite{VT}, Theorem 4.1.7 in \cite{VG2}), while $\msc{ES}$ is
properly contained in $\op{SD}_{\meet}$.  Hence $\msc{ES}$ is not a
quasivariety, which means it must not be closed under ultraproducts.
This result has been generalized in Kearnes and Nation~\cite{KN}.

\section{Connection with Quasivarieties}

In this section, we will show that for each quasivariety $\msc K$ of 
structures, the lattice of quasi-equational theories $\op{Qth}(\msc K)$ 
is isomorphic to the congruence lattice of a semilattice with operators.

Given a quasivariety $\msc K$, let $\F = \F_{\msc K}(\omega)$ be the 
$\msc K$-free algebra on $\omega$ generators, and let $\op{Con}_{\msc K}\,\F$
be the lattice of $\msc K$-congruences of $\F$.  
For a set $S$ of atomic formulae, recall that the $\msc K$-congruence
generated by $S$ is 
\[ \op{con}_{\msc K}\,S = \bigcap \{ \psi \in \op{Con} \,\F : 
   \F/\psi \in \msc K \text{ and } S \subseteq \psi \} .  \]
Then let $\T=\T_{\msc K}$ denote the join semilattice of compact 
$\msc K$-congruences in $\op{Con}_{\msc K}\,\F$.  
Thus $\T = (\op{Con}_{\msc K}\, \F_{\msc K}(\omega))_c$
consists of finite joins of the form $\Join_j \varphi_j$, with each
$\varphi_j$ either $\op{con}_{\msc K}\,(s,t)$ or $\op{con}_{\msc K}\,R(\ol s)$ 
for terms $s$, $t$, $s_i \in \F$ and a relation $R$.

Let $X$ be a free generating set for $\F_{\msc K}(\omega)$.  Any map
$\sigma_0 : X \to \F$ can be extended to an endomorphism $\sigma : \F \to \F$
in the usual way.  Since the image $\sigma(\F)$ is a substructure
of $\F$, the kernel of an endomorphism $\sigma$ is a 
$\msc K$-congruence.  The endomorphisms of $\F$ form a monoid $\op{End}\,\F$.

The endomorphisms of $\F$ act naturally on $\T$.  For
$\vare \in \op{End}\,\F$, define 
\begin{align*}
\widehat{\vare}(\op{con}_{\msc K}\,(s,t)) 
    &= \op{con}_{\msc K}\,(\vare s,\vare t) \\  
\widehat{\vare}(\op{con}_{\msc K}\,R(\ol s)) 
    &= \op{con}_{\msc K}\,R(\vare \ol s) \\  
\widehat{\vare}(\Join_j \varphi_j) &= \Join_j \widehat{\vare}\varphi_j.
\end{align*}
The next lemma is used to check the crucial technical details that
$\widehat{\vare}$ is well-defined, and hence join-preserving. 

\begin{lm} \label{check}
Let\/ $\msc K$ be a quasivariety, $\F$ a $\msc K$-free algebra, and
$\vare \in \op{End}\,\F$.  Let $\alpha$, $\beta_1, \dots, \beta_m$
be atomic formulae.  In\/ $\op{Con}_{\msc K}\,\F$,
\[  \op{con}_{\msc K} \,\alpha \leq \Join \op{con}_{\msc K}\, \beta_j
\qquad \text{implies} \qquad
 \widehat{\vare}(\op{con}_{\msc K}\,\alpha) \leq 
\Join \widehat{\vare}(\op{con}_{\msc K} \,\beta_j)  .\]
\end{lm}

\begin{proof}
For an atomic formula $\alpha$ and a congruence $\theta$, let us write 
$\alpha \in \theta$ to mean 
either (1) $\alpha$ is $s \approx t$ and $(s,t) \in \theta_0$, or
(2) $\alpha$ is $R(\ol s)$ and $\ol s \in \theta_1^R$.
So for the lemma, we are given that if $\F/\psi \in \msc K$ and
$\beta_1, \dots, \beta_m \in \psi$, then $\alpha \in \psi$.
We want to show that if $\F/\theta \in \msc K$ and
$\vare\beta_1, \dots, \vare\beta_m \in \theta$, then $\vare\alpha \in \theta$.

Let $\theta \in \op{Con}\,\F$ be a congruence such that $\F/\theta \in \msc K$,
and let $h:\F \to \F/\theta$ be the natural map.  
Then $h\vare : \F \to \F/\theta$, and since $h\vare(\F)$ is a substructure 
of $h(\F)$, the image is in $\msc K$.  
Now $\beta_1, \dots, \beta_m \in \ker h\vare$, and so $\alpha \in \ker h\vare$.
Thus $\vare \alpha \in \ker h = \theta$, as desired.
\end{proof}

Now let $\xi$ be a compact $\msc K$-congruence.  
Suppose that $\xi = \Join_i \varphi_i$ and $\xi = \Join_j \psi_j$ 
in $\T$, with each $\varphi_i$ and $\psi_j$ being a principal 
$\msc K$-congruence.  
Then for each $i$ we have $\varphi_i \leq \Join_j \psi_j$,
whence $\widehat{\vare}\varphi_i \leq \Join_j \widehat{\vare}\psi_j$
by Lemma~\ref{check}.  
Thus $\Join_i \widehat{\vare}\varphi_i \leq \Join_j \widehat{\vare}\psi_j$.
Symmetrically
$\Join_j \widehat{\vare}\psi_j \leq \Join_i \widehat{\vare}\varphi_i$, 
and so $\widehat{\vare} \xi = \Join_j \widehat{\vare}\psi_j = 
\Join_i \widehat{\vare}\varphi_i$ is well-defined.  

It then follows from the definition of $\widehat{\vare}$ that if
$\varphi = \Join_i \varphi_i$ and $\psi = \Join_j \psi_j$ in $\T$, then
\begin{align*}
\widehat{\vare}( \varphi \join \psi ) &=
   \widehat{\vare}(\Join_i \varphi_i \,\join\, \Join_j \psi_j )\\
&=\Join_i \widehat{\vare}\varphi_i \,\join\, \Join_j \widehat{\vare}\psi_j )\\
   &= \widehat{\vare}\varphi \join \widehat{\vare}\psi .
\end{align*}
Thus $\widehat{\vare}$ preserves joins.  
Also note that for the zero congruence we have $\widehat{\vare}(0)=0$.

Let $\widehat{\E}=\{ \widehat{\vare}: \vare \in \op{End}\,\F \}$, and 
consider the algebra $\S= \S_{\msc K}= \la \T, \join, 0, \widehat{\E} \ra$.
By the preceding remarks, 
the operations of $\widehat{\E}$ are \emph{operators} on $\S$, i.e.,
$(\join,0)$-homomorphisms, so $\S$ is a join semilattice with operators.
With this setup, we can now state our main result.

\begin{thm} \label{rep2}
For a quasivariety $\msc K$,
\[ \op{L}_q(\msc K) \cong^d \op{Con}\,\S \]
where $\S = \la \T,\join,0,\widehat{\E} \ra$ with
$\T$ the semilattice of compact congruences of\/ $\op{Con_{\msc K}}\,\F$, 
$\E = \op{End}\,\F$, and\/ $\mbf F = \mbf F_{\msc K}(\omega)$.
\end{thm}

In Part II, we will use this technical variation.

\begin{thm} \label{rep3}
Let $\msc K$ be a quasivariety and let $n \geq 1$ be an integer.
Then the lattice of all quasi-equational theories that
\begin{enumerate}
\item contain the theory of\/ $\msc K$, and
\item are determined relative to $\msc K$ by quasi-identities in
at most $n$ variables, 
\end{enumerate}
is isomorphic to $\op{Con}\,\S_n$, 
where $\S_n = \la \T_n,\join,0,\widehat{\E} \ra$ with 
$\T_n$ the semilattice of compact congruences of\/ 
$\op{Con_{\msc K}}\,\F$, 
$\E = \op{End}\,\F$, and\/ $\mbf F = \mbf F_{\msc K}(n)$.
\end{thm}

We shall prove Theorem~\ref{rep2}, and afterwards discuss the
modification required for Theorem~\ref{rep3}, which is essentially
just replacing $\F_{\msc K}(\omega)$ by $\F_{\msc K}(n)$.

For the proof of Theorem~\ref{rep2}, and for its application, it is natural 
to use two structures closely related to the congruence lattice instead.  
For an algebra $\A$ with a join semilattice reduct,
let $\op{Don}\,\A$ be the lattice of all reflexive, transitive, compatible
relations $R$ such that $\geq \,\subseteq R$, i.e., $x \geq y$ implies
$x \, R \, y$.
Let $\op{Eon}\,\A$ be the lattice of all reflexive, transitive, compatible
relations $R$ such that
\begin{enumerate}
\item $R \subseteq \leq$, i.e., $x \,R\, y$ implies $x \leq y$, and
\item if $x \leq y \leq z$ and $x \,R\, z$, then $x \,R\, y$.
\end{enumerate}

\begin{lm}
If\/ $\A = \la A, \join, 0, \msc F \ra$ is a semilattice with operators,
then\/ $\op{Con}\,\A \cong \op{Don}\,\A \cong \op{Eon}\,\A$.
\end{lm}

\begin{proof}
Let $\delta: \op{Con}\,\A \to \op{Don}\,\A$ \emph{via}
$\delta(\theta) = \theta \circ \geq$, so that
\[  x \,\delta(\theta)\, y \quad\text{iff}\quad x \,\theta \, x \join y \]
and let $\gamma: \op{Don}\,\A \to \op{Con}\,\A$ \emph{via}
$\gamma(R) = (R\, \cap \leq) \circ (R\, \cap \leq)^{\ssm}$, so that
\[  x \,\gamma(R)\, y \quad\text{iff}\quad x \,R\, x \join y \ \&\ 
    y \,R\, x \join y .  \]
Now we check that, for $\theta \in \op{Con}\,\A$ and $R \in \op{Don}\,\A$,
\begin{enumerate}
\item $\delta(\theta) \in \op{Don}\,\A$,
\item $\gamma(R) \in \op{Con}\,\A$,
\item $\delta$ and $\gamma$ are order-preserving,
\item $\gamma\delta(\theta)=\theta$,
\item $\delta\gamma(R)=R$.
\end{enumerate}
This is straightforward and only slightly tedious.

Similarly, let $\vare: \op{Don}\,\A \to \op{Eon}\,\A$ \emph{via}
$\vare(R)=R \,\cap \leq$, 
and $\delta': \op{Eon}\,\A \to \op{Don}\,\A$ \emph{via}
$\delta'(S)=S \,\circ \geq$, and check the analogous statements
for this pair, which is again routine.  Note that for a congruence
relation $\theta$ the corresponding eon-relation is
$\vare\delta(\theta) = \theta \,\cap \leq$, while for $S \in \op{Eon}\,\A$
we have the congruence $\gamma\delta'(S) = S \circ S^{\ssm}$.
\end{proof}

Now we define a Galois connection between $T^2$ and structures 
$\A \in \msc K$.
(The collection of structures $\A \in \msc K$ forms a proper class.
However, every quasivariety is determined by its finitely generated 
members.  So we could avoid any potential logical difficulties by 
restricting our attention to structures $\A$ defined on some fixed 
infinite set large enough to contain an isomorphic copy of each
finitely generated member of $\msc K$.)
For a pair $(\beta, \gamma) \in T^2$ and $\A \in \msc K$,
let $(\beta, \gamma) \,\Xi\, \A$ if, whenever $h:\F \to \A$ is a
homomorphism, $\beta \leq \ker h$ implies $\gamma \leq \ker h$.

Then, following the usual rubric for a Galois connection, 
for $X \subseteq T^2$ let
\[
\kappa(X) = \{ A \in \msc K: (\beta, \gamma) \,\Xi\, \A \text{ for all }
(\beta, \gamma) \in X \}.
\]
Likewise, for $Y \subseteq \msc K$, let
\[
\Delta(Y) = \{ (\beta, \gamma) \in T^2: (\beta, \gamma) \,\Xi\, \A
\text{ for all } \A \in Y \}.
\]
We must check that the following hold for $X \subseteq T^2$ and
$Y \subseteq \msc K$.

\begin{enumerate}
\item $\kappa(X) \in \op{L_q}(\msc K)$,
\item $\Delta(Y) \in \op{Don}\,\S$,
\item $\Delta\kappa(X)=X$ if $X \in \op{Don}\,\S$,
\item $\kappa\Delta(Y)=Y$ if $Y \in \op{L_q}(\msc K)$.
\end{enumerate}

To prove (1), we show that $\kappa(X)$ is closed under substructures, direct 
products and ultraproducts.  Closure under substructures is immediate, and closure 
under direct products follows from the observation that if $h:\F \to \prod_i \A_i$
then $\ker h = \bigcap \ker \pi_i h$.  So let $\A_i \in \kappa(X)$ for
$i \in I$, let $U$ be an ultrafilter on $I$, and let $h:\F \to
\prod \A_i/U$ be a homomorphism.  Since $\F$ is free, we can find
$f:\F \to \prod \A_i$ such that $h=gf$ where $g: \prod \A_i \to \prod \A_i/U$
is the standard map.  Let $(\beta,\gamma) \in X$ with
$\beta = \Join \varphi_j$ and $\gamma = \Join \psi_k$,
where these are finite joins and each $\varphi$ and $\psi$ is of the form
$\op{con}_{\msc K}\, \alpha$ for an atomic formula $\alpha$.  Each $\alpha$ in turn
is of the form either $s \approx t$ or $R(\ol s)$.

Assume $\beta \leq \ker h$.  Then $h(\alpha_j)$ holds for each $j$, so that
for each $j$ we have $\{ i \in I : \pi_i f(\alpha_j) \} \in U$.
Taking the intersection,
$\{ i \in I : \forall j \ \pi_i f(\alpha_j) \} \in U$.
In other words, $\{ i \in I: \beta \leq \ker \pi_i f \} \in U$,
and so the same thing holds for $\gamma$.  Now we reverse the steps
to obtain $\gamma \leq \ker h$, as desired.  Thus $\kappa(X)$ is also
closed under ultraproducts, and it is a quasivariety.

To prove (2), let $Y \subseteq \msc K$.  It is straightforward that
$\Delta (Y) \subseteq T^2$ is a relation that is reflexive, transitive,
and contains $\geq$.  Moreover, if $(\beta,\gamma) \in \Delta(Y)$ and
$\beta \join \tau \leq \ker h$ for an appropriate $h$, then
$\gamma \join \tau \leq \ker h$, so $\Delta(Y)$ respects joins.

Again let $(\beta,\gamma) \in \Delta(Y)$ and $h:\F \to \A$ with $A \in Y$.
Let $\widehat{\vare} \in \widehat{\E}$ and assume that $\widehat{\vare}\beta
\leq \ker h$.  This is equivalent to $\beta \leq \ker h\vare$, as both mean
that $h\vare(\alpha_j)$ holds for all $j$, where 
$\beta = \Join \op{con}_{\msc K}\, \alpha_j$.  
Hence $\gamma \leq \ker h\vare$, yielding $\widehat{\vare}\gamma \leq \ker h$. 
Thus $\Delta(Y)$ is compatible with the operations of
$\widehat{\E}$.  We conclude that $\Delta(Y) \in \op{Don}\,\S$.

Next consider (4).  Given that $Y$ is a quasivariety, we want to show that
$\kappa\Delta(Y) \subseteq Y$.  Let $\A \in \kappa\Delta(Y)$, and let
$\&_j \, \alpha_j \!\implies\! \zeta$ be any quasi-identity holding in $Y$.
Set $\beta = \Join \op{con}_{\msc K}\,\alpha_j$ and $\gamma = 
\op{con}_{\msc K}\,\zeta$,
and let $h:\F \to \A$ be a homomorphism.
Then $(\beta,\gamma) \in \Delta(Y)$, whence as $\A \in \kappa\Delta(Y)$
we have $\beta \leq \ker h$ implies $\gamma \leq \ker h$.
Thus $\A$ satisfies the quasi-identity in question, which shows that
$\kappa\Delta(Y) \subseteq Y$, as desired.

Part (3) requires the most care (we must show that relations in
$\op{Don}\,\S$ correspond to theories of quasivarieties).  Given
$X \in \op{Don}\,\S$, we want to prove that $\Delta\kappa(X) \subseteq X$.

Let $(\mu,\nu) \in T^2 - X$.  Define a congruence $\theta$ on $\F$ as
follows.
\begin{align*}
\theta_0 &= \mu \\
\theta_{k+1} &= \theta_k \join \Join \{ \gamma | (\beta,\gamma) \in X
\text{ and } \beta \leq \theta_k \} \\
\theta &= \Join_k \theta_k .
\end{align*}
Let $\C= \F/\theta$.  We want to show that $\C \in \kappa(X)$ and that
$\nu \nleq \theta$.

\emph{Claim a.  If $\psi$ is compact and $\psi \leq \theta$, then
$(\mu,\psi) \in X$.} We prove by induction that if compact
$\psi \leq \theta_k$, then $(\mu,\psi) \in X$.  For $k=0$ this is trivial.

Assume the statement holds for $k$.
Suppose we have a finite collection of $(\beta_i,\gamma_i) \in X$ with
each $\beta_i \leq \theta_k$.  Let $\xi=\Join \beta_i$, so that
$\xi$ is compact and $\beta_i \leq \xi \leq \theta_k$.  Then
$(\xi,\beta_i) \in X$, so by transitivity $(\xi,\gamma_i)\in X$ for all $i$.
Hence $(\xi,\Join \gamma_i) \in X$.  Now inductively $(\mu,\xi) \in X$,
and so $(\mu,\Join \gamma_i) \in X$.

\emph{Claim b.  If $(\beta,\gamma) \in X$ and $\beta \leq \theta$, then
$\gamma \leq \theta$.}  This holds by construction and compactness.

\emph{Claim c.  $\F/\theta \in \kappa(X)$.}  Suppose $h:\F \to \F/\theta$,
$(\beta,\gamma) \in X$ and $\beta \leq \ker h$.  Let $f:\F \to \F/\theta$
be the standard map with $\ker f = \theta$.  There exists an endmorphism
$\vare$ of $\F$ such that $h=f\vare$.  Then, using Claim b and an argument
above,
\begin{align*}
\beta \leq \ker h = \ker f\vare
      &\implies \widehat{\vare}\beta \leq \ker f = \theta \\
      &\implies \widehat{\vare}\gamma \leq \theta = \ker f \\
      &\implies \gamma \leq \ker f\vare = \ker h.
\end{align*}

\emph{Claim d.  $(\mu,\nu) \notin \Delta\kappa(X)$.}  This is because
$\C \in \kappa(X)$ by Claim c and $\mu \leq \theta = \ker f$, while
$\nu \nleq \theta$ by Claim a.

This completes the proof of (3), and hence Theorem~\ref{rep2}.

Only a slight modification is required for Theorem~\ref{rep3}.
Consider the collection of quasivarieties $\msc C$ satisfying the 
conditions of the theorem:
\begin{enumerate}
\item $\msc C \subseteq \msc K$, and
\item $\msc C$ is determined relative to $\msc K$ by quasi-identities in
at most $n$ variables. 
\end{enumerate}
These properties mean that a structure $\C$ is in $\msc C$ if and only if
\begin{enumerate}
\item[$(1)'$] Every map $f_0:\omega \to \C$ extends to a homomorphism 
$f: \F_{\msc K}(\omega) \to \C$, and
\item[$(2)'$] Every map $g_0:n \to \C$ extends to a homomorphism  
$g: \F_{\msc C}(\omega) \to \C$.
\end{enumerate}
Quasivarieties satisfying conditions (1) and (2) are closed under
arbitrary joins, and thus under containment they form a lattice which
we will denote by $\L_q^n(\msc K)$. 
This is a complete join subsemilattice of $\L_q(\msc K)$; the 
corresponding dual lattice of theories is a complete meet subsemilattice
$\op{QTh}^n(\msc K)$ of $\op{QTh}(\msc K)$.
The proof of Theorem~\ref{rep2} gives us $\op{QTh}(\msc K)$ as the 
congruence lattice of a semilattice with operators obtained from
$\F_{\msc K}(\omega)$.
In view of condition $(2)'$, the same construction with $\F_{\msc K}(\omega)$
replaced throughout by $\F_{\msc K}(n)$ yields $\op{QTh}^n(\msc K)$.

\section{Interpretation}

The foregoing analysis is rather structural and omits the motivation,
which we supply here.  Let $\beta$ and $\gamma$ be elements of $\T$,
i.e., compact $\msc K$-congruences on the free structure $\mbf F$.
Then these are finite joins in $\op{Con}_{\msc K}\,\F$ of
principal congruences, say $\beta = \Join \op{con}_{\msc K}\,\alpha_j$ and 
$\gamma = \Join \op{con}_{\msc K}\,\zeta_k$, where each $\alpha$ and $\zeta$
is an atomic formula of the form $s \approx t$ or $R(\ol s)$.
The basic idea is that the congruence $\op{con}(\beta,\beta \join \gamma)$,
on the semilattice $\S$ of compact $\msc K$-congruences of $\F$ with
the endomorphisms as operators, should correspond to the
conjunction over the indices $k$ of the quasi-identities
$\&_j \, \alpha_j \!\implies\! \zeta_k$, and that furthermore the
quasi-equational consequences of combining implications (modulo the theory
of $\msc K$) behaves like the join operation in $\op{Con}\,\S$.
But $\beta \geq \gamma$ should mean that $\beta \!\implies\! \gamma$,
so it is really $\op{Don}\,\S$ that we want.  On the other hand,
all the nontrivial information is contained already in $\op{Eon}\,\S$,
and these three lattices are isomorphic.

Let $H(\beta,\gamma)$ denote the set of all quasi-identities 
$\&_j \, \alpha_j \!\implies\! \zeta_k$ where the atomic formulae
$\alpha_j$ and $\zeta_k$ come from join representations
$\beta = \Join \op{con}_{\msc K}\,\alpha_j$ and 
$\gamma = \Join \op{con}_{\msc K}\,\zeta_k$. 
Let $\Delta$ and $\kappa$ be the mappings from the Galois connection
in the proof of Theorem~\ref{rep2}.
The semantic versions of the structural results of the preceding section
then take the following form.

\begin{lm} \label{L9}
Let $\msc Q$ be a quasivariety contained in $\msc K$.
The set of all pairs $(\beta,\gamma)$ such that $\msc Q$ satisfies
each of the sentences in $H(\beta,\gamma)$ is in $\op{Don}\,\S$,
where $\S = \la \T,\join,0,\widehat{\E} \ra$ with
$\T$ the semilattice of compact congruences of\/ $\op{Con_{\msc K}}\,\F$, 
$\E = \op{End}\,\F$, and\/ $\mbf F = \mbf F_{\msc K}(\omega)$.
\end{lm}

\begin{lm} \label{L10}
Let\/ $Y$ be a collection of structures contained in $\msc K$.  
The following are equivalent.
\begin{enumerate}
\item $(\beta,\gamma) \in \Delta(Y)$.
\item Every $\A \in Y$ satisfies all the implications in $H(\beta,\gamma)$.
\item The quasivariety $\op{SPU}(Y)$ satisfies all the implications
in $H(\beta,\gamma)$.
\end{enumerate}
\end{lm}

\begin{lm} \label{L11}
Let\/ $X \subseteq T^2$, where $\T$ is as in Lemma~\ref{L9}.  
The following are equivalent for a structure $\A$.
\begin{enumerate}
\item $\A \in \kappa(X)$.
\item For every pair $(\beta,\gamma) \in X$, $\A$ satisfies all the
quasi-identities of\/ $H(\beta,\gamma)$.
\end{enumerate}
\end{lm}

As always, it is good to understand both the semantic and logical viewpoint.

\section{Congruence lattices of semilattices with operators} \label{eqint}

Let us examine more closely lattices of the form $\op{Con}(\S,+,0,\msc F)$.
The following theorem summarizes some fundamental facts about their structure.

\begin{thm} \label{basic}
Let $(\S,+,0,\msc F)$ be a semilattice with operators.
\begin{enumerate}
\item An ideal\/ $I$ of\/ $\S$ is the $0$-class of some congruence relation
if and only if\/ $f(I) \subseteq I$ for every $f \in \msc F$.
\item If the ideal\/ $I$ is $\msc F$-closed, then the least congruence
with $0$-class $I$ is $\eta(I)$, the semilattice congruence generated by $I$.
It is characterized by 
\[ 
x \ \eta(I)\  y \qquad\text{iff}\qquad x+i=y+i \text{ for some } i \in I. 
\]
\item There is also a greatest congruence with $0$-class $I$, which we will
denote by $\tau(I)$.  To describe this, let $\msc F^{\dagger}$ denote the
monoid generated by $\msc F$, including the identity function.  Then 
\[
x \ \tau(I)\  y \qquad\text{iff}\qquad (\forall h \in \msc F^{\dagger}) 
\ h(x) \in I \ifff h(y) \in I.
\]
\end{enumerate}
\end{thm}

The proof of each part of the theorem is straightforward.
As a sample application, it follows that if $\S$ is a \emph{simple}
semigroup with one operator, then $|\S|=2$.

The maps $\eta$ and $\tau$ from Theorem~\ref{basic} induce operations 
on the entire congruence lattice $\op{Con}(\S,+,0,\msc F)$.  
If $\theta$ is a congruence with $0$-class $I$, 
define $\eta(\theta) = \eta(I)$ and $\tau(\theta) = \tau(I)$.
The map $\eta$ is known as the \emph{natural equa-interior operator} on
$\op{Con}(\S,+,0,\msc F)$.  This terminology will be justified below.

The natural equa-interior operator induces a partition of 
$\op{Con}(\S,+,0,\msc F)$.

\begin{thm} \label{equapartition}
Let\/ $\S = \la S,+,0,\msc F \ra$ be a semilattice with operators.
The natural equa-interior operator partitions $\op{Con}(\S)$ into intervals
$[ \eta(\theta),\tau(\theta) ]$ consisting of all the congruences with the
same $0$-class (which is an $\msc F$-closed ideal).
\end{thm}

The natural equa-interior operator on the congruence lattice of a 
semilattice with operators plays a role dual to that of the
equaclosure operator for lattices of quasivarieties.

Adaricheva and Gorbunov \cite{AG}, building on Dziobiak~\cite{WD},
described the natural equational closure operator on $Q$-lattices.
In the dual language of theories, the restriction of quasi-equational 
theories to atomic formulae gives rise to an equa-interior operator 
(defined below) on $\op{QTh}(\msc K)$.
Finitely based subvarieties of a quasi-variety $\msc K$ are given by 
quasi-identities that can be written as
$x \approx x \!\implies\! \&_k \, \beta_k$ for some atomic formulae $\beta_k$.  
By Lemma~\ref{L10}, the corresponding congruences are of the form 
$\op{con}(0,\theta)$ where $\theta$ is a compact $\msc K$-congruence 
on the free algebra $\F_{\msc K}(\omega)$.
More generally, subvarieties of $\msc K$ correspond to joins of these, i.e.,
to congruences of the form $\Join_{\theta \in I} \op{con}(0,\theta)$
for some ideal $I$ of the semilattice of compact $\msc K$-congruences.
Thus we should expect the map $\eta$ to be the analogous interior operator
on congruence lattices of semilattices with operators. 

We now define an equa-interior operator abstractly to have those
properties that we know to hold for the natural equa-interior operator
on the lattice of theories of a quasivariety.  One of our main goals, in
this section and the next two, is to extend this list of known properties
using the representation of the lattice of theories as the congruence
lattice of a semilattice with operators.

An \emph{equa-interior operator} on an algebraic lattice $\L$ is
a map $\eta: L \to L$ satisfying the following properties.
\begin{enumerate}
\item[(I1)] $\eta(x) \leq x$
\item[(I2)] $x \geq y$ implies $\eta(x) \geq \eta(y)$
\item[(I3)] $\eta^2(x)=\eta(x)$
\item[(I4)] $\eta(1)=1$
\item[(I5)] $\eta(x)=u$ for all $x \in X$ implies $\eta(\Join X)=u$
\item[(I6)] $\eta(x) \join (y \meet z) = (\eta(x) \join y) \meet (\eta(x) 
\join z)$
\item[(I7)] The image $\eta(\L)$ is the complete join subsemilattice of $\L$
generated by $\eta(\L) \cap \L_c$.
\item[(I8)] There is a compact element $w \in \L$ such that $\eta(w)=w$ and
the interval $[w,1]$ is isomorphic to the congruence lattice of a 
semilattice.  (Thus the interval $[w,1]$ is coatomistic.)
\end{enumerate}

Property (I5) means that the operation $\tau$ is implicitly defined by
$\eta$, \emph{via}  
\[ \tau(x) = \Join \{ z \in \L : \eta(z) = \eta(x) \} . \]
Thus $\tau(x)$ is the largest element $z$ such that $\eta(z)=\eta(x)$.
Likewise, properties (I1) and (I3) insure that 
$\eta(x)$ is the least element $z'$ such that $\eta(z')=\eta(x)$.
By (I2), if $\eta(x) \leq y \leq \tau(x)$, then $\eta(y)=\eta(x)$.
Thus the kernel of $\eta$,  defined by $x \approx y$ iff $\eta(x)=\eta(y)$,
is an equivalence relation that partitions $\L$ into disjoint intervals
of the form $[\eta(x),\tau(x)]$.  We will refer to this as the
\emph{equa-partition} of $\L$.

Now $\tau$ is not order-preserving in general.  However, 
it does satisfy a weak order property that can be useful.

\begin{lm} \label{ten}
Let\/ $\L$ be an algebraic lattice, and assume that $\eta$ satisfies conditions
(I1)--(I5).  Define $\tau$ as above.
Then for any subset $\{ x_j : j \in J \} \subseteq L$,
\[ 
\tau(\Meet_{j \in J} x_j) \geq \Meet_{j \in J} \tau(x_j).
\]
\end{lm}

\begin{proof}
We have
\[ \eta(\Meet \tau x_j) \leq \Meet \eta \tau x_j \leq \Meet x_j \leq 
\Meet \tau x_j \]
and that's all in one block of the equa-partition, while
$\Meet x_j \leq \tau (\Meet x_j)$, which is the top of the same block.
Thus $\Meet \tau x_j \leq \tau (\Meet x_j)$.
\end{proof}


Property $(I7)$ has some nice consequences.

\begin{lm} \label{re7}
Let $\eta$ be an equa-interior operator on an algebraic lattice $\L$.
\begin{enumerate}
\item The image $\eta(\L)$ is an algebraic lattice, and
$x$ is compact in $\eta(\L)$ iff $x \in \eta(\L)$ and $x$ is compact in $\L$.
\item If $X$ is up-directed, then $\eta(\Join X) = \Join \eta(X)$.
\end{enumerate}
\end{lm}

For any quasivariety $\msc K$, the natural equa-interior operator on the
lattice of theories of $\msc K$ satisfies the eight listed basic properties.  
Congruence lattices of semilattices with operators come close.
For an ideal $I$ in a semilattice with operators, let $\op{con}_{\op{SL}}(I)$ 
denote the semilattice congruence generated by collapsing all the elements 
of $I$ to $0$.

\begin{thm} \label{equaint}
If\/ $\S = \la S, +, 0, \msc F \ra$ is a semilattice with operators,
then the map $\eta$ on $\op{Con}\,\S$ given by
$\eta(\theta) = \op{con}_{\op{SL}}(0/\theta)$ 
satisfies properties (I1)--(I7).  
\end{thm}

\begin{proof}
Property (I6) is the hard one to verify.  Let $\alpha$, $\beta$, $\gamma
\in \op{Con}\,\S$ and let $\xi = \eta(\alpha)$.  Then $x \,\xi\, y$ if
and only if there exists $z \in S$ such that $z \,\alpha\, 0$ and
$x + z = y + z$.  (This is the semilattice congruence but it's
compatible with $\msc F$.)  We want to show that
\[ (\xi \join \beta) \meet (\xi \join \gamma) \leq
\xi \join (\beta \meet \gamma). \]
Let $a$, $b \in \op{LHS}$.  Then there exist elements such that
\begin{align*}
a \,\beta\, c_1 \,\xi\, c_2 \,\beta\, c_3 &\dots b \\
a \,\gamma\, d_1 \,\xi\, d_2 \,\gamma\, d_3 &\dots b.
\end{align*}
Let $z$ be the join of the elements witnessing the above $\xi$-relations.
Then
\[ a \,\xi\, a + z \,\beta\, c_1 + z = c_2 + z \,\beta\,
c_3 + z = \dots b + z \,\xi\, b \]
so that $a \,\xi\,a + z \,\beta\, b + z \,\xi\, b$, and similarly
$a \,\xi\,a + z \,\gamma\, b + z \,\xi\, b$.  Thus $a$, $b \in
\op{RHS}$, as desired.
\end{proof}

Property (I8), on the other hand, need not hold in the congruence lattice
of a semilattice with operators.  The element $w$ of (I8),
called the \emph{pseudo-one}, in lattices of quasi-equational theories
corresponds to the identity $x \approx y$.  For an equa-interior operator
on a lattice $\L$ with $1$ compact, we can take $w=1$; in particular,
this applies when the semilattice has a top element, in which case we can 
take $w = \op{con}(0,1)$.  But in general, there may be no candidate for 
the pseudo-one.

Note that property (I8) implies that a lattice is dually atomic 
(or \emph{coatomic}).  Let $x<1$ in $\L$.  If $x \join w < 1$ then
it is below a coatom, while if $x \join w = 1$ then by the compactness
of $w$ there is a coatom above $x$ that is not above $w$.
In particular, the lattice of theories of a quasivariety is coatomic
(Corollary~5.1.2 of Gorbunov~\cite{VG2}).

Consider the semilattice $\mbf \Omega = (\omega,\join,0,p)$
with $p(0)=0$ and $p(x)=x-1$ for $x>0$.  Then $\op{Con}\,\mbf \Omega \cong
\omega + 1$, which has no pseudo-one (regardless of how $\eta$ is defined).
Thus $\op{Con}\,\mbf \Omega$ is not the dual of a $Q$-lattice. 
Likewise, $\op{Con}\,\mbf \Omega$ fails to be dually atomic.

In each of the next two sections we will discuss an additional property
of the natural equa-interior operator on semilattices with operators.
The point of this is that an algebraic lattice cannot be the dual of
a $Q$-lattice unless it admits an equa-interior operator satisfying
all these conditions.   Indeed, we should really consider the representation
problem in the context of pairs $(\L,\eta)$, rather than just the representation
of a lattice with an unspecified equa-interior operator.

For the sake of clarity, let us agree that the term \emph{equa-interior 
operator} refers to conditions (I1)--(I8) for the remainder of the paper,
even though we are proposing that henceforth a ninth condition 
should be included in the definition.

\section{A new property of natural equa-interior operators}
\label{eleven}

The next theorem gives a property of the natural equa-partition
on congruence lattices of semilattices with operators that need not hold 
in all lattices with an equa-interior operator.

\begin{thm} \label{prop}
Let\/ $\S = \la S,+,0,\msc F \ra$ be a semilattice with operators,
and let $\eta$, $\tau$ denote the bounds of the natural equa-partition
on $\op{Con}\,\S$.   If the congruences $\zeta$, $\gamma$, $\chi$
satisfy $\eta(\zeta) \leq \eta(\gamma)$ and $\tau(\chi) \leq \tau(\gamma)$,
then
\[ \eta (\eta(\zeta) \join \tau (\zeta \meet \chi)) \leq \eta(\gamma) .\]
\end{thm}

\begin{proof}
Assume that $\zeta$, $\gamma$, $\chi$ satisfy the hypotheses, and let
$0/\zeta = Z$, $0/\gamma=C$ and $0/\chi=X$ be the corresponding ideals.
So $Z \subseteq C$ and $\tau(X) \subseteq \tau(C)$.
For notation, let $\alpha=\tau(Z \cap X)$.

We want to show that
$0/(\eta(Z) \join \alpha) \subseteq C$,
so let $w \in \op{LHS}$.  For any $z \in Z$ we have
$(z,w) \in \eta(Z) \join \alpha$.
Fix an element $z_0 \in Z$.
We claim that there exist elements $z^* \in Z$ and $w^* \in S$ such that
$z_0 \leq z^* \leq w^*$, $w \leq w^*$ and  $z^* \,\alpha\, w^*$.

There is a sequence
\[  z_0 = s_0  \ \eta(Z)\ s_1 \ \alpha\ s_2 \ \eta(Z)\ s_3 \ \dots\  s_k=w. \]
Let $t_j = s_0 + \dots +s_j$ for $0 \leq j \leq k$.  Thus we obtain
\[  z_0 = t_0  \ \eta(Z)\ t_1 \ \alpha\ t_2 \ \eta(Z)\ t_3 \ \dots\  t_k \]
with
\[  t_0 \leq t_1 \leq t_2 \leq t_3 \leq \dots \leq t_k.  \]
Put $z'=t_1$ and $w'=t_k$, so that with $z_0 \leq z' \in Z$ and $w \leq w'$.  
Moreover, we may assume that $k$ is minimal for such a sequence.

If $k>2$, then
$z' = t_1 \ \alpha\ t_2 \ \eta(Z)\ t_3 \ \alpha\ t_4$.  By the definition
of $\eta(Z)$, there exists $u \in Z$ such that $t_2+u = t_3+u$.
Joining with $u$ yields the shorter sequence
\[  z'' = t_1+u \ \alpha \ t_2+u = t_3+u \ \alpha\ t_4+u \ \dots \]
contradicting the minimality of $k$.
Thus $k \leq 2$, which yields the conclusion of the claim
with $z^* = t_1$ and $w^* = t_2$.  

Next, we claim that $(z^*,w^*) \in \tau(X)$.  This follows from the sequence
of implications:
\begin{align*}
f(z^*) \in X &\implies f(z^*) \in X \cap Z \\
             &\implies f(w^*) \in X \cap Z \\
             &\implies f(w^*) \in X \\
             &\implies f(z^*) \in X
\end{align*}
which hold for any $f \in \msc F$, using the $\msc F$-closure of $Z$,
$(z^*,w^*) \in \tau(X \cap Z)$ and $z^* \leq w^*$.

Thus $(z^*,w^*) \in \tau(X) \subseteq \tau(C)$.
But $z^* \in Z \subseteq C = 0/\tau(C)$, whence $w^* \in C$ and $w \in C$,
as desired.
\end{proof}

\begin{figure}[htbp]
\begin{center}
\includegraphics[height=2.0in,width=6.0in]{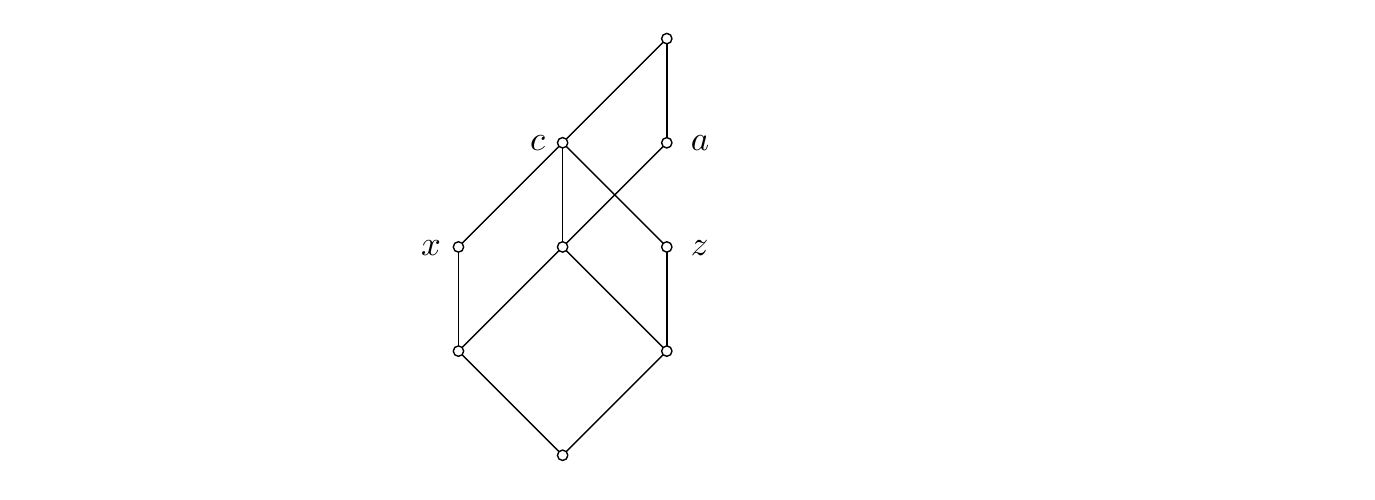}
\caption{$\mbf K$}{\label{figx9}}
\end{center}
\end{figure}

For an application of this condition, consider the lattice $\mbf K$ in 
Figure \ref{figx9}.
It is straightforward to show that $\mbf K$ has a unique equa-interior
operator, with $h(t)=0$ if $t \leq a$ and $h(t)=t$ otherwise.
Indeed, any equa-interior operator on $\mbf K$ must have
$h(a) \join (x \meet z) = (h(a) \join x) \meet (h(a) \join z)$, from
which it follows easily that $h(a)=0$.  But then we cannot have
$h(x)=0$, else $h(1) = h(a \join x) = 0$, a contradiction.  Thus
$h(x)=x$ and symmetrically $h(z)=z$.  This in turn yields that
$h(c)=c$.

But $\mbf K$ is not the congruence lattice of a semilattice
with operators.  The only candidate for the equa-interior operator
fails the condition of Theorem~\ref{prop} with the substitution
$\zeta \mapsto z$, $\gamma \mapsto c$, $\chi \mapsto x$.
Therefore $\mbf K$ is not the lattice of theories of a quasivariety.
We could have also derived this latter fact by noting that $\mbf K$
is not dually biatomic: in $\mbf K$ we have $a \geq x \meet z$ which 
is not refinable to a meet of coatoms.

On the other hand, $\mbf K$ can be represented as a filterable sublattice 
of $\op{Con}(\B_3,+,0)$, where $\B_3$ is the Boolean lattice on three atoms.  
(See Appendix~II for this terminology.)  Indeed, if
the atoms of $\B_3$ are $p$, $q$, $r$ then we can take
\begin{align*}
a &\mapsto  [0] \ [p,q,r,p \join q, p \join r, q \join r,1] \\
c &\mapsto \op{con}(0,p \join q)  \\
x &\mapsto \op{con}(0,p)  \\
z &\mapsto \op{con}(0,q)  .
\end{align*}
We will pursue the comparison of congruence lattices and lattices of
algebraic sets in the appendices.

Taking a cue from this example, we continue investigating the
consequences of the condition of Theorem~\ref{prop}.
Recall that, whenever $\eta$ satisfies (I1)--(I5), we have
$\eta(y)=\eta(x)$ iff $\eta(x) \leq y \leq \tau(x)$.  
The condition can be written as follows, where 
we use the fact that $\eta(u) \leq c$ iff $\eta(u) \leq \eta(c)$.
\[ (\dagger) \quad \tau(x) \leq \tau(c) \ \&\ \eta(z) \leq c \implies
\eta(\eta(z) \join \tau(x \meet z)) \leq c   \]
This holds for the natural equa-interior operator on congruence lattices
of semilattices with operators, and we want to see how it applies to pairs
$(\L,h)$ where $h$ is an arbitrary equa-interior operator on $\L$.

There is a two-variable version of the condition, which
is obtained by putting $c = \eta(z) \join \tau(x)$.
\[ (\ddagger) \qquad \eta(\eta(z) \join \tau(x \meet z)) \leq
   \eta(z) \join \tau(x)   \]
This appears to be slightly weaker than $(\dagger)$.

Consider the Boolean lattice $\B_3$ with atoms $x$, $y$, $z$
and the equa-interior operator with $h(y)=0$ and $h(t)=t$ otherwise.
Then $(\B_3,h)$ fails the condition $(\ddagger)$, though 
$\B_3$ is a dual $Q$-lattice with another equa-interior operator  
by Theorem~\ref{one-elt}.

There are two additional conditions on equa-interior operators
that are known to hold in the duals of $Q$-lattices:  bicoatomicity 
and the four-coatom condition.  (See Section~5.3 of Gorbunov \cite{VG2}.)
Unfortunately, congruence lattices of
semilattices with operators need not be coatomic
(there is an example in the discussion of property (I8) in Section~\ref{eqint}),
but duals of $Q$-lattices are, so we will impose this as an extra condition.
In that case, we will see that $(\dagger)$ implies both of these properties.

A lattice $\L$ is \emph{bicoatomic} (or \emph{dually biatomic}) 
if whenever $p$ is a coatom of $\L$ and $p \geq u \meet v$ properly,
then there exist coatoms $c \geq u$ and $d \geq v$ such that $p \geq c \meet d$.

\begin{thm} \label{bicoatom}
Let\/ $\L$ be a coatomic lattice and let $h$ be an equa-interior
operator on $\L$.  If\/ $(\L,h)$ satisfies property $(\dagger)$, then $\L$
is bicoatomic.
\end{thm}

\begin{proof}
Assume $1 \succ p \geq u \meet v$ properly in $\L$.
We want to find elements $c$, $z$ with $1 \succ c \geq u$, $z \geq v$,
and $c \meet z \leq p$.  (Then apply the argument a second time.)

Note that $p \geq \eta(p) \join (u \meet v) = (\eta(p) \join u) \meet
(\eta(p) \join v)$.  Put $x = \eta(p) \join u$ and $z=\eta(p) \join v$.
Let $1 \succ c \geq \tau(x)$ and note $\tau(x) \geq x \geq u$.

Suppose $c \meet z \nleq p$.  Put $z' = c \meet z$.  Then $\eta(z') \nleq p$,
for else since $\eta(p) \leq z'$ we would have $\eta(z') = \eta(p) =
\eta(z' \join p) = \eta(1) = 1$, a contradiction.  Now we apply $(\dagger)$.
Surely $\tau(x) \leq c$ and $\eta(z') \leq z' \leq c$.  Moreover
$\eta(p) \leq z' \meet x \leq z \meet x \leq p$ whence 
$\eta(z' \meet x)=\eta(p)$, and thus $\tau(z' \meet x)=p$.  But then
$\eta(\eta(z') \join \tau(x \meet z')) = \eta(\eta(z') \join p) = \eta(1)=1$,
again a contradiction.  Therefore $c \meet z \leq p$, as desired.
\end{proof}

The dual of the four-coatom condition played a significant role in
the characterization of the atomistic, algebraic $Q$-lattices.
This too is a consequence of property $(\dagger)$.   For coatoms $a$, $d$ we
write $a \sim d$ to indicate that $|\uparrow(a \meet d)|=4$, in which case 
the filter $\uparrow(a \meet d)$ is exactly $\{ 1,a,d,a \meet d \}$.
A lattice $\L$ with an equa-interior operator $\eta$ satisfies the 
\emph{four-coatom condition} if, whenever $a$, $b$, $c$, $d$ are coatoms
of $\L$ such that $a \sim d$, $\eta(a) \nleq d$, $\eta(c) \leq d$ and 
$\eta(c) = \eta(a \meet b)$, then $\eta(c) = \eta(b \meet d)$.

\begin{thm} \label{h7}
The four-coatom condition holds in a lattice with an 
equa-interior operator $\eta$ satisfying $(\dagger)$.  
\end{thm}

\begin{proof}
As $\eta(c) \leq b$, $d$ is given, we need that $\eta(b \meet d) \leq c$.
Supposing not, substitute $x=a \meet d$, $z = \eta(b \meet d)$, and the
element $d$ into $(\dagger)$.  Note that $\tau(a \meet d) \ne a$ has
$\eta(a) \nleq d$.  Thus $\tau(a \meet d) \leq d$, and of course
$\eta(b \meet d) \leq d$.  But we also have $\eta(c) \leq a \meet b \meet d
\leq a \meet b$ and $\eta(a \meet b) = \eta(c)$, so
$\eta(\eta(b \meet d) \join \tau(a \meet b \meet d)) =
 \eta(\eta(b \meet d) \join c) = \eta(1) = 1$, a contradiction.
Thus $\eta(b \meet d) \leq c$, as desired.
\end{proof}

\section{Coatomistic congruence lattices and a stronger property}
\label{coatomistic}

One of the most intriguing hypotheses about lattices of quasivarieties 
is formulated for atomistic lattices. Dually, it can be expressed as follows:
\begin{quote}
\emph{Can every coatomistic lattice of quasi-equational theories be represented 
as $\op{Con}(\S,+,0)$, i.e., without operators?}
\end{quote}
This hypothesis is shown to be valid in the case when the lattice of 
quasi-equational theories is dually algebraic \cite{ADG}.
The problem provides a motivation for investigating which coatomistic lattices
can be represented as lattices of equational theories, or congruence lattices of
semilattices, with or without operators.

Consider the class $\mathcal{M}$ of lattices dual to $\op{Sub}_f\,\M$, where 
$\M$ is an infinite semilattice with $0$, and $\op{Sub}_f\,\M$ is the lattice
of finite subsemilattices of $\M$, topped by the semilattice $\M$ itself.

Evidently, lattices in $\mathcal{M}$ are coatomistic, and they are algebraic
but not dually algebraic. Besides, it is straightforward to show that they 
cannot be presented as $\op{Con}(\S,+,0)$.
Thus, it would be natural to ask whether such lattices can be presented as 
$\op{Con}(\S,+,0,\msc F)$, for a non-empty set of operators on $\S$.
In many cases the answer is ``no" simply because there might be no 
equa-interior operator.  For example, let $\M$ be a meet semilattice such
that the dual of $\op{Sub}_f\,\M$ admits an equa-interior operator.
If $a$ is an element of $\M$ that can be expressed as a meet in
infinitely many ways, then $\eta(a)=0$ by Lemma~\ref{june2} below.
Hence $\M$ can contain at most one such element.

It turns out to be feasible to show that certain lattices from $\mathcal{M}$,
that \emph{do} admit an equa-interior operator, still cannot be represented as 
$\op{Con}(\S,+,0,\msc F)$.
The crucial factor here is to understand the behavior of infinite meets 
of coatoms, or more generally infinite meets of elements $\tau(x)$, in the 
congruence lattice of a semilattice with operators.  The restriction 
given by Theorem~\ref{twelve} can be expressed as a ninth basic property 
of the natural equa-interior operator (as it implies $(\dagger)$).

{\bf Aside:}  Coatoms arise naturally in another context, that does not
make the lattice coatomistic.  Suppose $\S = \la S,+,0,\msc F \ra$ has
the property that for each $\msc F$-closed ideal $I$, every $f \in \msc F$,
and every $x \in S$,
\[  f(x) \in I \implies x \in I .\]
Then the congruence $\tau(I)$ partitions $S$ into $I$ and $S-I$, and hence
is a coatom.  In particular, this property holds whenever
\begin{itemize}
\item $\msc F$ is empty, or
\item $\msc F$ is a group, or
\item every $f \in \msc F$ is increasing, i.e., $x \leq f(x)$ for all $x \in S$.
\end{itemize}
In all these cases, $\tau(\theta)$ is a coatom for every 
$\theta \in \op{Con}\,\S$.  We will be particularly concerned with the
case when $\msc F$ is a group in Part~II \cite{LQET2}.

\begin{thm} \label{twelve}
Let $\S = \la S,+,0,\msc F \ra$ be a semilattice with operators,
$I$ an arbitrary index set, and
$\chi$, $\gamma$, and $\zeta_i$ for $i \in I$ congruences on $\S$.
The natural equa-interior operator on $\op{Con}\,\S$ has the following
property:  if\/ $\eta(\chi) \leq \gamma$ and 
$\Meet_{i \in I} \tau(\zeta_i) \leq \tau(\gamma)$, then
\[  
\eta(\eta(\chi) \join \Meet_{i \in I}\tau( \chi \meet \zeta_i )) \leq \gamma . 
\]
\end{thm}

For the proof, it is useful to write down abstractly the two parts of the 
argument of the proof of Theorem~\ref{prop}.  

\begin{lm} \label{june4}
Let $\alpha$, $\chi$, $\zeta \in \op{Con}(\S,+,0,\msc F)$ and let 
$X$ be the $0$-class of $\chi$.
\begin{enumerate}
\item If\/ $u \in X$ and $(u,v) \in \chi \join \alpha$, then there exist
elements $u^*$, $v^*$ with $u \leq u^* \in X$, $v \leq v^*$, $u^* \leq v^*$,
and $(u^*,v^*) \in \alpha$.
\item If\/ $u \in X$, $u \leq v$ and $(u,v) \in \tau(\chi \meet \zeta)$,
then $(u,v) \in \tau(\zeta)$.
\end{enumerate}
\end{lm}

Now, under the assumptions of the theorem, let $u \in X$ and 
$(u,v) \in \eta(\chi) \join \Meet \tau( \chi \meet \zeta_i )$,
so that $v$ is in the $0$-class of the LHS.
Then by Lemma~\ref{june4}(1), there exist $u^*$, $v^*$ with 
$u \leq u^* \in X$, $v \leq v^*$, $u^* \leq v^*$ and
$(u^*,v^*) \in \Meet \tau( \chi \meet \zeta_i )$.
Then $(u^*,v^*) \in  \tau( \chi \meet \zeta_i )$ for every $i$,
whence by Lemma~\ref{june4}(2) $(u^*,v^*) \in \tau(\zeta_i)$ for every $i$,
so that $(u^*,v^*) \in \Meet \tau(\zeta_i)$.  

Let $X$ and $C$ denote the $0$-classes of $\chi$ and $\gamma$, 
respectively.  By assumption, we have $u^* \in X \subseteq C$,
and $(u^*,v^*) \in \Meet \tau(\zeta_i) \leq \tau(\gamma)$, so $v^* \in C$
as well.  \emph{A fortiori}, $v \in C$, as desired.

This proves Theorem~\ref{twelve}.  Thus we obtain the ninth fundamental
property of the natural equa-interior operator on the congruence lattice 
of a semilattice with operators.
\begin{enumerate}
\item[(I9)]  For any index set $I$, if $\eta(x) \leq c$ and 
$\Meet \tau(z_i) \leq \tau(c)$, then
$\eta(\eta(x) \join \Meet_{i \in I}\tau( x \meet z_i )) \leq c$. 
\end{enumerate}
As before, there is also a slightly simpler (and weaker) variation:
\[
(I9') \qquad \eta(\eta(x) \join \Meet_{i \in I}\tau( x \meet z_i )) \leq 
\eta(x) \join \Meet \tau(z_i) .
\]

Clearly, if $|I|=1$ then property (I9) reduces to property $(\dagger)$.  
In fact, for $I$ finite, $(\dagger)$ implies (I9).  
But for $I$ infinite, property (I9) seems to carry a rather different sort
of information, as we shall see below.

Consider the case when $|I|=2$; the argument for the general finite case
is similar.   Assume that $\eta(x) \leq c$ and $\tau(y) \meet \tau(z) \leq
\tau(c)$.   Using (I6), $(\dagger)$, and the fact that $\eta(u \meet v) =
\eta(\eta(u) \meet \eta(v))$, we calculate
\begin{align*}
\eta(\eta(x) \join (\tau(x \meet y) \meet \tau(x \meet z))) &=
\eta((\eta(x) \join (\tau(x \meet y)) \meet (\eta(x) \join \tau(x \meet z))))\\
&\leq  \eta((\eta(x) \join (\tau(y)) \meet (\eta(x) \join \tau(z))))\\
&= \eta(\eta(x) \join (\tau(y) \meet \tau(z)))\\
&\leq c 
\end{align*}
as desired.

With property (I9) as a tool-in-hand, we turn to a
thorough investigation of the (dual) dependence relation for
coatoms of $\op{Con}(\S,+,0,\msc F)$; see Theorems~\ref{june5} 
and~\ref{june6} below.
Throughout the remainder of this section, $\chi$, $\zeta$ and $\alpha$
will denote distinct coatoms of the congruence lattice.
Repeatedly, we use the basic property of equa-interior operators that
$\eta x \join (y \meet z) = (\eta x \join y) \meet (\eta x \join z)$.
Our goal is to generalize (to whatever extent possible) the following
property of finite sets of coatoms.

\begin{thm} \label{june1}
Let\/ $\L$ be a lattice with an equa-interior operator.
If for coatoms $x, z_1, \dots, z_k, a_1, \dots, a_k$ of\/ $\L$ we have
$x \meet z_i \leq a_i$ properly, then $\eta x \join \Meet_{i=1}^k z_i =1$.
\end{thm}

The proof uses the next lemma.

\begin{lm} \label{june2}
Suppose $x \meet z \leq a$ properly for coatoms in a lattice with an
equa-interior operator.  Then $\eta a \leq x \meet z$, and thus
\begin{enumerate}
\item $\tau(x \meet z) = a$,
\item $\eta x \nleq a$,
\item $\eta x \nleq z$.
\end{enumerate}
\end{lm}

\begin{proof}
If say $\eta a \nleq x$, then $\eta a \join x =1$, and using (I6) we would have
\begin{align*}
\eta a \join z &= (\eta a \join x) \meet (\eta a \join z) \\
               &= \eta a \join (x \meet z) \leq a
\end{align*}
whence $z \leq a$, a contradiction.  So $\eta a \leq x$, and symmetrically
$\eta a \leq z$.  Since $\eta a \leq x \meet z \leq a = \tau a$, we have
$\tau(x \meet z) = a$.

It follows that we cannot have $\eta x \leq a$, else
\[ \eta a = \eta (x \meet z) \leq \eta x \leq a, \]
implying that $\eta x = \eta a$, and thus 
$\eta a = \eta(x \join a) = \eta 1 = 1$ by (I5) and (I4), a contradiction.
Therefore also $\eta x \nleq z$, else $\eta x \leq x \meet z \leq a$. 
\end{proof}

The theorem now follows immediately, because
\[ \eta x \join \Meet_{i=1}^k z_i = \Meet_{i=1}^k (\eta x \join z_i) =1 .\]

The property of Theorem~\ref{june1} can fail when there are infinitely
many $z_i$'s, even in the congruence lattice of a semilattice.
Let $\mbf Q$ be the join semilattice in Figure~\ref{figx12}.
Consider the ideals
\begin{align*}
X &= \{ 0,u_1,u_2,u_3,\dots \} \\
Z_i &= \id v_i \\
A_i &= \id u_i
\end{align*}
for $i \in \omega$, and let $\chi=\tau(X)$, $\zeta_i = \tau(Z_i)$ and
$\alpha_i = \tau(A_i)$.
Then an easy calculation shows that $\Meet \zeta_i = 0$, and the infinite
version of the property of the theorem fails.

\begin{figure}[htbp]
\begin{center}
\includegraphics[height=2.8in,width=6.0in]{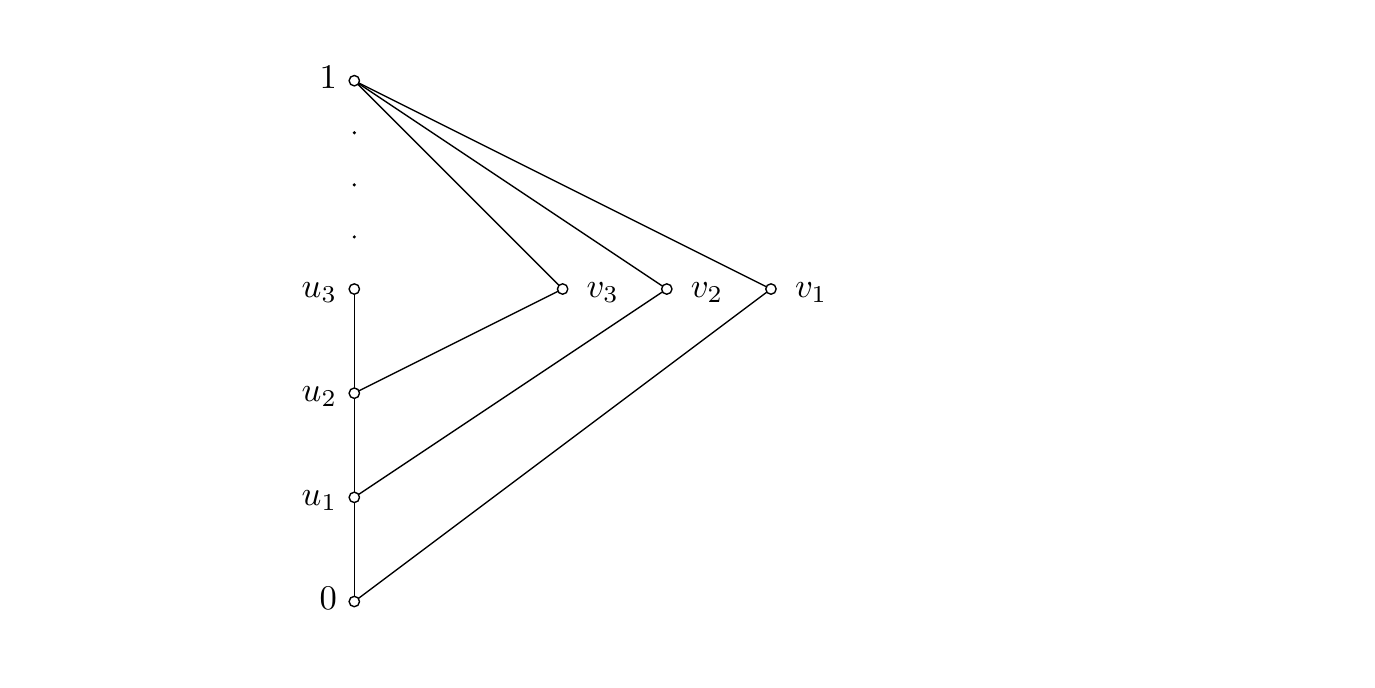}
\caption{$\op{Con}(\S,+,0)$ does not satisfy the infinite analogue
of Theorem~\ref{june1}.}{\label{figx12}}
\end{center}
\end{figure}

Nonetheless, we shall show that a couple of infinite versions do hold.

\begin{thm} \label{june5}
Let\/ $\L$ be a lattice with an equa-interior operator satisfying
property (I9).
If for coatoms $a$, $x$ and $z_i$ $(i \in I)$ of\/ $\L$ we have
$x \meet z_i \leq a$ properly, then $\eta x \join \Meet_{i \in I} z_i =1$.
\end{thm}

\begin{proof}
By Lemma~\ref{june2}, we have $\tau(x \meet z_i) = a$ for every $i$,
and $\eta x \nleq a$.  Hence $\eta x \join \Meet \tau (x \meet z_i) = 1$.
Then property $(I9')$ gives the conclusion immediately.
\end{proof}


\begin{thm} \label{june6}
Let\/ $\L$ be a lattice with an equa-interior operator satisfying
property (I9).  Let $x$, $a_i$ and $z_i$ be coatoms of\/ $\L$ with
$x \meet z_i \leq a_i$ properly for all $i \in I$.
If\/ $\Meet_{i \in I} a_i \nleq x$, then $\Meet_{i \in I} z_i \nleq x$.
\end{thm}

\begin{proof}
Again, by Lemma~\ref{june2}, we have $\tau(x \meet z_i) = a_i$ for every $i$.
Now apply (I9) directly with $c=x$. 
\end{proof}


Let us now use these results to show that certain coatomistic lattices
are not lattices of quasi-equational theories.
Call an infinite ($\meet$)-semilattice $\M$ \emph{cute} if it has an
element $a$ and different elements $m, m_j \in M\backslash\{a\}$, 
$j \in \omega$, with $m \wedge m_j=a$.

Examples of cute semilattices are $\M_\infty$: countably many $m_i$ covering 
the least element $a$, or $\M_2$:  a chain $\{m_j, j \in \omega\}$ in addition 
to elements $m,a$,  satisfying $m \wedge m_j = a$ for all $j$. It was asked 
in \cite{Ad} (p.~175), in connection with the hypothesis about the atomistic 
$Q$-lattices mentioned above in the dual form, whether
$\op{Sub}_f\,\M_\infty$ is a $Q$-lattice. The following result,
an immediate application of Theorem~\ref{june5}, answers this 
question in the negative.
   
\begin{thm}\label{cute}
If\/ $\M$ is a cute semilattice, then the dual of\/ $\op{Sub}_f\,\M$ is not 
representable as $\op{Con}(\S,+,0,\msc F)$.
Hence $\op{Sub}_f\,\M$ is not a $Q$-lattice.
\end{thm}

It would be desirable to extend Theorem \ref{cute} to all lattices from 
$\mathcal{M}$.  In particular, we may ask about possibility to represent 
$\L=(\op{Sub}_f\,\mathbf{P_1})^d$, where the semilattice $\mathbf{P_1}$ 
consists of two descending chains 
$\{b_i, i \in \omega\}$, $\{a_i,i \in \omega\}$ with 
defining relations $a_{i+1}=a_i \meet b_{i+1}$, $b_0 > a_0$.

Every equa-interior operator $\eta$ on $\L$ would satisfy: 
$\eta(\{a_i\})=[a_i,b_0]$, $\eta(\{b_i\})\geq [b_i,b_0]$. In particular, 
$\eta(c)=0$, $c \in \L$, implies $c=0$ (equivalently, $\tau(0)=0$). This 
makes $\mathbf{P_1}$ drastically different from cute semilattices.
\emph{Is the dual of\/ $\op{Sub}_f\,\mathbf{P_1}$ representable as 
$\op{Con}(\S,+,0,\msc F)$?}

Another interesting case to consider would be $\op{Sub}_f\,\C$ where
$\C$ is an infinite chain, so that every finite subset of $C$ is a 
subsemilattice.

\section{Appendix I: Complete sublattices of subalgebras}

In the first two appendices, we analyze conditions that were used in
older descriptions of lattices of quasivarieties; see Gorbunov~\cite{VG2}.

Note that $\op{Con}(\S,+,0,\msc F)$ is a complete sublattice of 
$\op{Con}(\S,+,0)$,
which is dually isomorphic to $\op{Sp}(\msc I(\S))$, which is the lattice
of subalgebras of an infinitary algebra.  (Joins of non-directed sets can
be set to $1$.)  In this context we are considering complete sublattices
of $\op{Sub}\,\A$ where $\A$ is a semilattice, or a complete semilattice,
or a complete algebra of algebraic subsets.  

Let $\vare$ be a binary relation on a set $S$.  A subset $X \subseteq S$ 
is said to be \emph{$\vare$-closed} if $c \in X$ and $c \,\vare\, d$ 
implies $d \in X$.

Recall that a quasi-order $\vare$ on a semilattice $\S = \la S,\meet,1 \ra$
is \emph{distributive} if it satisfies the following conditions.
\begin{enumerate}
\item If $c_1 \meet c_2 \,\vare\, d$ then there exist elements $d_1$, $d_2$
such that $c_i \,\vare\, d_i$ and $d = d_1 \meet d_2$.
\item If $1 \,\vare\, d$ then $d=1$.
\end{enumerate}

The effect of the next result is that for a semilattice $\S$, \emph{any} 
complete sublattice of $\op{Sub}\,\S$ can be represented as the lattice
of all $\rho$-closed subsemilattices, for some distributive quasi-order $\rho$.

\begin{thm} \label{oldnine}
Let\/ $\S= \la S,\meet,1 \ra$ be a semilattice with $1$, and let $\vare$
be a distributive quasi-order on $\S$.  Then $\op{Sub}\,(\S,\vare)$,
the lattice of all $\vare$-closed subsemilattices (with $1$), is a complete
sublattice of\/ $\op{Sub}\,\S$.

Conversely, let $\T$ be a complete sublattice of\/ $\op{Sub}\,\S$.
Define a relation $\rho$ on $\S$ by $c \,\rho\, d$ if for all $\mbf X \in \T$
we have $c \in X \!\implies\! d \in X$.
Then $\rho$ is a distributive quasi-order, and
$\T$ consists precisely of the $\rho$-closed subsemilattices of\/ $\S$. 
Furthermore, $\rho$ satisfies the following conditions.
\begin{enumerate}
\item [(3)] If $c \,\rho\, d_1$, $d_2$ then $c \,\rho\, d_1 \meet d_2$.
\item [(4)] For all $c \in S$, $c \,\rho\, 1$.
\end{enumerate}

The correspondence between complete sublattices of\/ $\op{Sub}\,\S$
and distributive quasi-orders satisfying (3) and (4) is a dual isomorphism.
\end{thm}

The proof is relatively straightforward.


The description of all complete sublattices of $\op{Sub}\,\S$, 
the lattice of all complete subsemilattices of a complete semilattice $\S$,
is almost identical, except that complete meets appear in the conditions.
\begin{enumerate}
\item[(1)${}'$] If $\Meet c_i \,\vare\, d$ then there exist elements $d_i$
such that $c_i \,\vare\, d_i$ and $d = \Meet d_i$.
\item[(3)${}'$] If $c \,\vare\, c_i$ for all $i$, then
$c \,\vare\, \Meet c_i$.
\end{enumerate}
Complete semilattices satisfying (1)${}'$ are called \emph{Brouwerian} by
Gorbunov~\cite{VG2}.  The results can be summarized thusly.

\begin{thm} \label{CSLnine}
Let\/ $\S= \la S,\meet,1 \ra$ be a complete semilattice.
Then there is a dual isomorphism between complete sublattices
of\/ $\op{Sub}\S$ and quasi-orders satisfying conditions (1)${}'$, (2),
(3)${}'$ and (4).
\end{thm}

For complete sublattices of $\op{Sp}(\A)$, the lattice of algebraic
subsets of an algebraic lattice $\A$, we must also deal with joins
of nonempty up-directed subsets, and once $\A$ fails the ACC matters
get more complicated.  A quasi-order $\vare$ on $\A$ is said to be
\emph{continuous} if it has the following property.
\begin{enumerate}
\item[(5)] If $C$ is a directed set and $\Join C \,\vare\, d$, then
there exists a directed set $D$ such that $d = \Join D$ and for each
$d \in D$ there exists $c \in C$ with $c \,\vare\, d$.  
\end{enumerate}
This is a very slight weakening of Gorbunov's definition \cite{VG2}.  
As above, we have this result of Gorbunov.

\begin{thm}\label{alg}
Let $\vare$ be a continuous Brouwerian quasi-order on a complete lattice
$\A$.  Then $\op{Sp}(\A)$, the lattice of $\vare$-closed algebraic subsets,
is a complete sublattice of $\op{Sp}(\A)$.
\end{thm}

Now for any algebra $\B$ we can define the \emph{embedding relation} $E$
on $\op{Con}\,\B$ by $\theta \,E\, \psi$ if $\B/\psi \leq \B/\theta$.
A fundamental result of Gorbunov characterizes $Q$-lattices in terms of
the embedding relations (Corollaries~5.2.2 and~5.6.8 of \cite{VG2}).

\begin{thm}\label{fund}
Let $\msc K$ be a quasivariety and let $\F = \F_{\msc K}(\omega)$.
The embedding relation is a continuous Brouwerian quasi-order on
$\op{Con}_{\msc K}\,\F$, and
$\L_q(\msc K) \cong \op{Sp}(\op{Con}_{\msc K}(\F,E))$.
\end{thm}

For comparison, we note that the isomorphism relation need not be
continuous; see Gorbunov \cite{VG2}, Example~5.6.6.

We do not know (and doubt) that the relation $\rho$ corresponding 
to a complete sublattice of $\op{Sp}(\A)$ need always be continuous.  
However, our representation of $\op{Con}(\S,+,0,\msc F)$ as dually 
isomorphic to a complete sublattice of $\op{Sp}(\msc I(\S))$ could be
unraveled to give the $\rho$ relation explicitly in that case.
\emph{Are these particular relations always continuous?} 

\section{Appendix II: Filterability and equaclosure operators}

The natural equational closure operator on $L_q(\msc K)$ is given 
by the map $h(\msc Q) = \op{H}(\msc Q) \cap \msc K$ for quasivarieties
$\msc Q \subseteq \msc K$.
That is, $h(\msc Q)$ consists of all members of $\msc K$ that are in
the variety generated by $\msc Q$, or equivalently, that are homomorphic
images of $\F_{\msc Q}(X)$ for some set $X$.  
For the corresponding map on $\op{Sp}(\op{Con}\,\F_{\msc K}(\omega))$, 
let $X$ be the algebraic subset of all $\msc Q$-congruences of
$\op{Con}\,\F_{\msc K}(\omega)$.  Then $\varphi = \Meet X$ is the 
natural congruence with $\F/\varphi \cong \F_{\msc Q}(\omega)$, and 
the filter $\fil \varphi$ is the algebraic subset associated with $h(\msc Q)$,
that is, all $h(\msc Q)$-congruences of $\op{Con}\,\F_{\msc K}(\omega)$.

Abstractly, let $\vare$ be a distributive quasi-order on an algebraic lattice
$\A$.  Then it is not hard to see that
the map $h(X)=\fil \Meet X$ on $\op{Sp}(\A,\vare)$
will satisfy the duals of conditions (I1)--(I7) so long as 
$\fil \Meet X$ is $\vare$-closed for every $X \in \op{Sp}(\A,\vare)$.
A quasi-order that satisfies this crucial condition,
\[ c \geq \Meet X \ \&\  c \,\vare\, d \!\implies\! d \geq \Meet X \]
is said to be \emph{filterable}.
If the quasi-order $\vare$ is filterable, then the closure operator 
$h(X)=\fil \Meet X$ on $\op{Sp}(\A,\vare)$ is again called the \emph{natural} 
closure operator determined by $\vare$.
We can also speak of a complete sublattice of $\op{Sp}(\A)$ as being
filterable if the quasi-order it induces \emph{via} Theorem~\ref{oldnine} 
is so.

Dually, a sublattice $\T \leq \op{Con}(\S,+,0)$ is filterable if,
for each $\theta \in \T$, the semilattice congruence generated by 
the $0$-class of $\theta$ is in $\T$.  
As we have observed, this is the case when $\T = \op{Con}(\S,+,0,\msc F)$
for some set of operators $\msc F$.
Thus we obtain a slightly different perspective on Theorem~\ref{equaint}.

\begin{thm} \label{filterable}
For a semilattice $\S$ with operators, $\T = \op{Con}(\S,+,0,\msc F)$
is a filterable complete sublattice of\/ $\op{Con}(\S,+,0)$. Thus $\T$ 
supports the natural interior operator $h(\theta)=\op{con}(0/\theta)$, 
which satisfies conditions (I1)--(I7).
\end{thm}

In fact, the natural interior operator on $\op{Con}(\S,+,0,\msc F)$
also satisfies condition (I9).
However, as we saw in Section~\ref{eleven}, a filterable sublattice of
$\op{Con}(\S,+,0)$ may fail condition $(\dagger)$, which is the finite 
index case of (I9), even with $\S$ finite.
Thus being a congruence lattice of a semilattice with operators is a
stronger property than just being a filterable sublattice of 
$\op{Con}(\S,+,0)$.

\section{Appendix III: Lattices of equational theories}

In this appendix, we summarize what is known about lattices of equational
theories.  Throughout the section, $\V$ will denote a variety of algebras,
with no relation symbols in the signature.  For this situation, atomic 
theories really are equational theories.  The lattice of equational theories 
is, of course, dual to the lattice of subvarieties of $\V$.

From the basic representation 
$\op{ATh}(\V) \cong \op{Ficon} \, \F_{\V}(\omega)$,
we see that the lattice is algebraic.  Its top element $1$ has the basis
$x \approx y$, and thus $1$ is compact.  On the other hand, J.~Je\v zek
proved that any algebraic lattice with countably many compact elements
is isomorphic to an interval in some lattice of equational theories
\cite{JJ}.  

R. McKenzie showed that every lattice of equational theories is isomorphic
to the congruence lattice of a groupoid with left unit and right zero
\cite{RMcK}.  N.~Newrly refined these ideas, showing that a lattice of
equational theories is isomorphic to the congruence lattice of a monoid
with a right zero and one additional unary operation \cite{NN}.
A.~Nurakunov added a second unary operation and proved a converse:
a lattice is a lattice of equational theories if and only if it is the
congruence lattice of a monoid with a right zero and two unary operations
satisfying certain properties~\cite{AMN}.

Nurakunov's conditions are rather technical, but they just codify the
properties of the natural operations on the free algebra $\F_{\V}(X)$
that they model.  If $X = \{ x_0, x_1, x_2, \dots \}$ and $s$, $t$
are terms, then
\[  s \cdot t = t(s,x_1, x_2, \dots ). \]
The two unary operations are the endomorphism $\varphi_+$ and $\varphi_-$,
where $\varphi_+(x_i)=x_{i+1}$ for all $i$, while $\varphi_-(x_0) = x_0$ 
and $\varphi_-(x_i)=x_{i-1}$ for $i>0$.  

W.A. Lampe used McKenzie's representation to prove that lattices of equational
theories satisfy a form of meet semidistributivity at $1$, the so-called
Zipper Condition \cite{WAL}:
\[ \text{If }a_i \meet c = z \text{ for all }i \in I \text{ and }
\Join_{i \in I} a_i = 1, \text{ then } c=z.   \]
A similar but stronger condition was found by M.~Ern\'e \cite{ME} and 
G.~Tardos (independently), which was refined yet further by Lampe \cite{WAL2}.
These results show that the structure of lattices of equational
theories is quite constrained at the top, whereas Je\v zek's theorem 
shows that this is not the case globally.
Confirming this heuristic, D.~Pigozzi and G.~Tardos proved that every algebraic
lattice with a completely join irreducible greatest element $1$ is isomorphic 
to a lattice of equational theories \cite{PT}.

Again, we propose that one should investigate $\op{ATh}(\V)$ for varieties
of structures.  

The authors would like to thank the referee for many helpful comments.

\end{document}